\documentclass[11pt]{amsart}

\usepackage{amsmath,amsfonts,amsthm,amscd,amssymb,graphicx}
\numberwithin{equation}{section}

\newtheorem{theorem}{Theorem}[section]

\newtheorem{lemma}[theorem]{Lemma}

\newtheorem{proposition}[theorem]{Proposition}

\newtheorem{definition}[theorem]{Definition}

\def\eps{\varepsilon }
\def\D{\partial }

\newcommand{\RR}{\mathbb{R}}
\newcommand{\cO}{\mathcal{O}}

\newcommand{\cX}{\mathcal{X}}

\newcommand{\dt}{\frac{d}{dt}}

\newcommand{\ZZ}{{\mathbb Z}}
\newcommand{\TT}{{\mathbb T}}

\newcommand{\e}{{\varepsilon}}

\def\bb1{{1\!\!1}}

\def\cL{\mathcal{L}}

\def \tu{\tilde{u}}
\def \tv{\tilde{v}}

\def\D{\partial}
\def\dt{\partial_t}
\def\dx{\partial_x}
\def\dy{\partial_y}

\def\dY{\partial_Y}

\newcommand{\T}{{\mathbb T}}
\newcommand{\pa}{{\partial}}

\newcommand{\bmat}{\begin{pmatrix}}
\newcommand{\emat}{\end{pmatrix}}

\newcommand{\bitem}{\begin{itemize}}
\newcommand{\eitem}{\end{itemize}}

\def\myskip{\bigskip\noindent}

\begin{document}

\title[A note on the Prandtl boundary layers]{A note on the Prandtl boundary layers}

\author[Y. Guo, T. Nguyen]{Yan Guo, Toan Nguyen}

\address{Division of Applied Mathematics, Brown University, Providence, RI, USA}
\email{Yan$\underline{~}$Guo@Brown.edu;\;Toan$\underline{~}$Nguyen@Brown.edu}


\begin{abstract} This note concerns a nonlinear ill-posedness of the Prandtl equation and an invalidity of asymptotic boundary-layer expansions of incompressible fluid flows near a solid boundary. Our analysis is built upon recent remarkable linear ill-posedness results established by G\'erard-Varet and Dormy \cite{GD}, and an analysis in Guo and Tice \cite{GT}. We show that the asymptotic boundary-layer expansion is not valid for non-monotonic shear layer flows in Sobolev spaces. We also introduce a notion of weak well-posedness and prove that the nonlinear Prandtl equation is not well-posed in this sense near non-stationary and non-monotonic shear flows.  On the other hand, we are able to verify that Oleinik's monotonic solutions are well-posed. 

\end{abstract}

\date{Last updated: \today}

\maketitle

\tableofcontents

\section{Introduction}
One of classical problems in fluid dynamics is the vanishing viscosity limit of Navier-Stokes solutions near a solid boundary. To describe the problem, let us consider the two-dimensional incompressible Navier-Stokes equations:
\begin{equation}\label{NSeqs} \begin{aligned}\dt \bmat u^\nu\\v^\nu\emat + (u^\nu \dx + v^\nu \dy) \bmat u^\nu\\v^\nu\emat + \nabla p^\nu &= \nu \Delta \bmat u^\nu\\v^\nu\emat \\ \dx u^\nu+ \dy v^\nu &=0.\end{aligned}
\end{equation}
Here, $(x,y)\in \TT \times \RR_+$ and $(u^\nu,v^\nu)\in \RR\times \RR$ are the tangential and normal components of the velocity, respectively, corresponding to the boundary $y=0$. We impose the no-slip boundary conditions: $(u^\nu,v^\nu)_{|y=0} =0.$
A natural question is how one relates solutions of the Navier-Stokes equations to those of the Euler equations (i.e., equations \eqref{NSeqs} with $\nu=0$) with boundary condition $v^0\vert_{y=0}=0$ in the zero viscosity limit? Formally, 
one may expect an asymptotic description as follows:
\begin{equation}\label{exp-nu}\bmat u^\nu\\v^\nu\emat (t,x,y) = \bmat u^0\\v^0\emat(t,x,y) + \bmat u_p\\\sqrt \nu v_p\emat(t,x,y/{\sqrt\nu})\end{equation}
where $(u^0,v^0)$ solves the Euler equation and $(u_p,v_p)$ is the boundary layer correction that describes the transition near the boundary from zero velocity $u^\nu$ of the Navier-Stokes flow to the potentially nonzero velocity $u^0$ of the Euler flow and thus plays a significant role in the thin layer with order $\cO(\sqrt\nu)$. We may also express the pressure $p^\nu$ as $$p^\nu (t,x,y) = p^0(t,x,y) + p_p(t,x,\frac y{\sqrt\nu}).$$
We then can formally plug these formal Ansatz into \eqref{NSeqs} and derive the boundary layer equations for $(u_p,v_p)$ at the leading order in $\sqrt\nu$.  For our convenience, we denote $Y = y/\sqrt\nu$ and define $$\begin{aligned}u(t,x,Y) &:= u^0(t,x,0) + u_p(t,x,Y),\\
v(t,x,Y) &:= \dy v^0(t,x,0) Y + v_p(t,x,Y).\end{aligned}$$The boundary layer or Prandtl equation for $(u,v)$ then reads:  
\begin{equation}\label{prandtl}
\left\{
\begin{array}{rlll} 
\D_t u  +  u \D_x u +  v \D_Y  u - \D^2_Y  u  + \D_x P &=& 0,  \quad Y > 0,\\
\D_x  u + \D_Y v & = & 0,  \quad  Y > 0, \\
u\vert_{t=0} &=&u_0(x,y)\\
u \vert_{Y=0}= v\vert_{Y=0}  &=& 0, \\ 
\lim_{Y\rightarrow+\infty}  u & = & U(t,x),   \\  
\end{array}
\right.
\end{equation} where $U = u^0(t,x,0)$ and $P=P(t,x)$ are the normal velocity and pressure describing the Euler flow just outside the boundary layer, and satisfy the Bernoulli equation 
$$\D_t U + U\D_xU + \D_x P =0. $$

\myskip 
This formal idea was proposed by Ludwig Prandtl \cite{Pra:1904} in 1904 to describe the fluid flows near the boundary. Mathematically, we are interested in the following two problems:

\bitem

\item well-posedness of the Prandtl equation \eqref{prandtl};

\item rigorous justification of the asymptotic boundary layer expansion.   

\eitem

\myskip
Sammartino and Caflisch \cite{SC} resolved these issues in an analytic setting where the initial data and the outer Euler flow are assumed to be analytic functions. Oleinik \cite{Ole} established the existence and uniqueness of the Cauchy problem \eqref{prandtl} in a monotonic setting where the initial and boundary data are assumed to be monotonic in $y$ along the boundary-layer profile. For further mathematical results, see the review paper \cite{E:2000}.  In this paper, we address the above issues in a Sobolev setting. Our work is based on a recent result of G\'erard-Varet and Dormy \cite{GD} where they established ill-posedness for the Cauchy problem of the linearized Prandtl equation around non-monotonic shear flows. 

\myskip

In what follows, we shall work with the Euler flow which is constant on the boundary, that is, $U \equiv const.$ Also, by a shear flow to the Prandtl, we always mean that a special solution to \eqref{prandtl} has a form of $(u_s,0)$ with $u_s= u_s(t,Y)$. Thus, $u_s$ solves the heat equation:
\begin{equation}\label{shear-profile}
\left\{
\begin{array}{rlll} 
\D_t u_s&=& \dY^2 u_s,  \qquad Y > 0,\\
u_s\vert_{t=0} &=& U_s,
\end{array}
\right.
\end{equation} with initial shear layer $U_s$, and with the same boundary conditions at $Y=0$ and $Y=+\infty$ as in \eqref{prandtl}.

\myskip

We shall work on the standard Sobolev spaces $L^2$ and $H^m$, $m\ge 0$, with usual norms:
$$\|u\|_{L^2_{x,Y}}: = \Big(\int_{ \TT\times \RR_+} |u|^2 dxdY\Big)^{1/2} \quad \mbox{and}\quad \|u\|_{H^m_{x,Y}}:= \sum_{k=0}^m \sum_{i+j = k}\|\dx^i \dY ^ju\|_{L^2}.$$
For initial data, we will often take them to be in a weighted $H^m$ Sobolev spaces. For instance, we say $u_0\in e^{-\alpha Y}H^m_{x,Y}$ if $e^{\alpha Y}u_0 \in H^m_{x,Y}$ and has a finite norm, for some $\alpha>0$ (see, for example, \cite{GD,GVN} where this type of weighted spaces is used for initial data). We occasionally drop the subscripts ${x,Y}$ in $H^m_{x,Y}$ when no confusion is possible, and write $H_\alpha^m$ to refer to the weighted space $e^{-\alpha Y}H^m_{x,Y}$.  


%

\myskip

To state our results precisely, we introduce the following definition of well-posedness; here, we say that $u$ belongs to $U+ \cX$, for some functional space, to mean that $u-U\in \cX$.


\begin{definition}[Weak well-posedness]\label{def-wellposed} For a given Euler flow $u^0,$ denote $U(t,x)=u^0(t,x,0).$ We
say {\bf the Cauchy problem \eqref{prandtl} is locally weak well-posed} if there exist positive continuous functions $
T(\cdot,\cdot)$, $C(\cdot ,\cdot )$, some $\alpha>0$, and some integer $m\geq 1$ such that for any initial data $ u^1_0,u_0^2$ in $ U+
e^{-\alpha Y}H^{m}_{x,Y}(\T \times \RR_+)$, 
there are unique distributional solutions $u_1,u_2$ of \eqref{prandtl} in $U+ L^\infty(]0,T[; H^1_{x,Y}(\T \times \RR_+))$ with initial data ${u_j}\vert_{t=0} = u^j_0$, $j=1,2$, and there holds
\begin{equation}\label{stab-ineq} 
\begin{aligned}\sup_{0\leq t\leq T}&\|u_{1}(t)-u_{2}(t)\|_{H_{x,Y}^{1}}
\\&\leq C(\|e^{\alpha
Y}[u_0^1-U]\|_{H_{x,Y}^{m}},\|e^{\alpha Y}[u_0^2-U]\|_{H_{x,Y}^{m}})\|e^{\alpha Y}[u_0^1-u_0^2]\|_{H_{x,Y}^{m}},
\end{aligned}
\end{equation}
in which $T = T(\|e^{\alpha
Y}[u_0^1-U]\|_{H_{x,Y}^{m}},\|e^{\alpha Y}[u_0^2-U]\|_{H_{x,Y}^{m}})$.
\end{definition}

We note that when we choose $u_{2}\equiv 0$ in the above definition, we obtain an estimate for solutions in the $H_{x,Y}^{1}$ space. We call such a well-posedness weak because we allow the initial data to be in $H^m_{x,Y}$ for sufficiently large $m$.



\myskip
Our first main result then reads
\begin{theorem}[No Lipschitz continuity of the flow] \label{theo-illposed} The Cauchy problem \eqref{prandtl} is not  locally 
weak well-posed in the sense of Definition \ref{def-wellposed}. 
 \end{theorem}

\myskip

Our result is an improvement of a recent result obtained by D. G\'erard-Varet and the second author \cite{GVN} without additional sources in the Prandtl equation. 
In Section \ref{sec-Ole}, we will show that in the monotonic framework of Oleinik (see Assumption (O) in Section \ref{sec-Ole}), the Cauchy problem \eqref{prandtl} is well-posed in the sense of Definition \ref{def-wellposed}. The key idea is to use the Crocco transformation to obtain certain energy estimates for $\D_xu$. We note that as shown in \cite{GD}, the ill-posedness in the non-monotonic case is due to high-frequency in $x$ and the lack of control on $\D_x u$ in the original coordinates in \eqref{prandtl}. 

\myskip

Finally, regarding the validity of the asymptotic boundary layer expansion, we ask whether one can write 
\begin{equation}\label{non-expansion}
\bmat u^\nu\\v^\nu\emat (t,x,y) = \bmat u^0-u^0\vert_{y=0}\\ 0\emat (y)+\bmat u_{s}\\ 0\emat(t,\frac y{\sqrt\nu}) + (\sqrt\nu)^{\gamma} \bmat \tu^\nu\\ \sqrt \nu \tv^\nu\emat(t,x,\frac y{\sqrt\nu}),\end{equation}
and 
$$p^\nu (t,x,y) = (\sqrt\nu)^{\gamma}\tilde p^\nu(t,x,\frac y{\sqrt\nu}), $$
for shear flows $u_s$ and for some $\gamma>0$, where $(u^0(y),0)^{t}$ is the Euler flow. Our second main result asserts that this is false in general, for all $\gamma>0$. Again, to state our result precisely, we introduce the following definition of validity of the asymptotic expansion in Sobolev spaces. 
 
 \begin{definition}[Validity of asymptotic expansions]\label{def-valid-expansion} For a given Euler flow $u^0 = u^0(y)$, denote $U = u^0(0)$. We say {\bf the expansion \eqref{non-expansion} is valid with a $\gamma>0$} if there exist positive continuous functions $
T_\gamma(\cdot,\cdot)$, $C_\gamma(\cdot ,\cdot )$, some $\alpha>0$, and some integers $m'\ge m\geq 1$ such that for any initial shear layer $U_s$ in $U+e^{-\alpha Y}H^{m'}_Y(\RR_+)$ and any initial data $\tu_0,\tv_0$ in $e^{-\alpha Y}H^{m}_{x,Y}(\T \times \RR_+)$,  we can write \eqref{non-expansion} in $L^\infty([0,T_\gamma]; H^1_{x,Y}(\T \times \RR_+))$ with $u_s(0) = U_s$, $(\tu^\nu,\tv^\nu)\vert_{t=0}=(\tu_0,\tv_0)$, and $\tilde p^\nu \in  L^\infty([0,T_\gamma]; L^2_{x,Y}(\T \times \RR_+))$, and there holds 
\begin{eqnarray*}
\sup_{0\leq t\leq T_\gamma} \|(\tilde{u}^\nu(t),\tilde{v}^\nu
(t))\|_{H_{x,Y}^{1}} \leq C_{\gamma }(\|e^{\alpha Y}(U_s-U)\|_{H_{Y}^{m'}}, \|e^{\alpha Y}(
\tu_0,\tv_0)\|_{H_{x,Y}^{m}}),
\end{eqnarray*}
in which $T_\gamma = T_{\gamma }(\|e^{\alpha Y}(U_s-U)\|_{H_{Y}^{m'}}, \|e^{\alpha Y}(
\tu_0,\tv_0)\|_{H_{x,Y}^{m}})$. 
  \end{definition}

Our second main result then reads
 \begin{theorem}[Invalidity of asymptotic expansions] \label{theo-noexpansion} The expansion \eqref{non-expansion} is not valid in the sense of Definition \ref{def-valid-expansion} for any $\gamma>0$. 
\end{theorem}

\def\U{{\bf U}}

We will prove this theorem via contradiction. We show that the expansion does not hold for a sequence of translated 
shear layers $u_{s_n}(t) = u_{s_0}(t+s_n)$, $s_n$ being arbitrarily small, in which the initial shear layer $u_{s_0}(0)$ has a non-degenerate critical point as in \cite{GD}. Hence, if
the Olenick monotone condition is violated, our result indicates that the remainder $\tilde{u}^\nu(t),
\tilde{v}^\nu(t)$ in the asymptotic expansion (1.6) can not be bounded in terms of the initial data in a reasonable fashion.

Our result is different from Grenier's result [4] on invalidity of
asymptotic expansions.  He allows the initial perturbation data to be
arbitrarily small of size $\nu^n$ to a fixed unstable Euler shear flow
(could even be monotone!), and shows that in a very short time of size 
$\sqrt\nu \log(1/\nu)$, the solution u grows rapidly to $\cO(\nu^{1/4})$ in $L^\infty$. In contrast, we have to work with a family of
non-monotone shear flow profiles, and we do not know how badly the
solutions grow. On the other hand,  we show that the expansion is invalid in
order $\cO(\nu^{\gamma })$, for any $\gamma > 0$. In addition,
the blow-up norm in \cite{Gre:2000} is $H^1$ in the original variable y, whereas our result
concerns the (weaker) norm in the stretched variable $Y = y/\sqrt\nu$, upon noting that $\|u(y/\sqrt\nu)\|_{L^2_y(\RR_+)} = \nu^{1/4}\|u(Y)\|_{L^2_Y(\RR_+)} $.

We remark that both our ill-posedness and no expansion theorems are in terms of a weighted $H^1$ norm, which does not even control $L^\infty$ norm in the two-dimensional domain. However, our proofs fail for a weaker space than $H^1$ (e.g., $L^2$), because we need the local compactness of $H^1$ to pass to various limits as $\nu\to 0$.



%
  
  \myskip
  
  \section{Linear ill-posedness}\label{sec-linear}
 In this section, we recall the previous linear ill-posedness results obtained by G\'erard-Varet and Dormy \cite{GD} that will be used to prove our nonlinear illposedness. For notational simplicity, we define
the linearized Prandtl operator $\cL_s$ around a shear flow $u_s$:
$$\cL_s u := -\dY^2 u + u_s\dx u+ v\dY u_s, \qquad  v = -\int_0^Y\dx u dy'.$$

\myskip
With our notation, the nonlinear Prandtl equation \eqref{prandtl} in the perturbation variable $\tilde u := u - u_s$ then reads (dropping the titles):    
\begin{equation}\label{prandtl-pert}
\left\{
\begin{array}{rlll} 
\D_t u +\cL_s u &=& -u\dx u - v\dY u,  \qquad Y > 0,\\
u\vert_{t=0} &=& u_0,
\end{array}
\right.
\end{equation} with zero boundary conditions at $Y=0$ and $Y=\infty$. 

\myskip
Removing the nonlinear term in \eqref{prandtl-pert}, we call the resulting equation as the linearized Prandtl equation around the shear flow $u_s$:
\begin{equation}\label{prandtl-lin}
\D_t u +\cL_s u\;=\;0, \qquad u\vert_{t=0} \;=\;u_0.
\end{equation}
Denote by $T(s,t)$ the linearized solution operator, that is, 
$$T(s,t)u_0 := u(t)$$ where $u(t)$ is the solution to the linearized equation with $u\vert_{t=s} = u_0$. The following ill-posedness result is for the linearized equation \eqref{prandtl-lin}. 
\begin{theorem}[\cite{GD}]\label{theo-GD} There exists an initial shear layer $U_s$  to \eqref{shear-profile} which has a non-degenerate critical point such that for all $\e_0 >0$ and all $m\ge0$, there holds
\begin{equation}\label{GD-illposed} \sup_{0\le s\le t\le \e_0} \|T(s,t)\|_{\cL(H_\alpha^m,L^2)} = +\infty,
\end{equation}
where $ \|\cdot\|_{\cL(H_\alpha^m,L^2)}$ denotes the standard operator norm in the functional space $\cL(H_\alpha^m,L^2)$ consisting of linear operators from the weighted space $H_\alpha^m = e^{-\alpha Y}H^m$ to the usual $L^2$ space.
\end{theorem}
\begin{proof}[Sketch of proof] In fact, the instability estimate \eqref{GD-illposed} stated in \cite{GD} was from the weighted space $H_\alpha^m$ to another weighted space $H_\alpha^{m'}$. From their construction, \eqref{GD-illposed} remains true when the targeting space is not weighted. We thus sketch their proof where it applies to the usual $L^2$ space as stated.  We recall that the main ingredient in the proof is their construction of approximate growing solutions $u^\e$ to \eqref{prandtl-lin} such that for all small $\e$, $u^\e$ solves $$\dt u^\e + \cL_s u^\e = \e^M r^\e,$$ for arbitrary large $M$, where $u^\e$ and $r^\e$ satisfy:
$$ ce^{\theta_0 t/\sqrt\e}\le  \|u^\e(t)\|_{L^2} \le C e^{\theta_0 t/\sqrt\e}, \qquad \|e^{\alpha Y} r^\e(t)\|_{H^m }\le C\e^{-m}e^{\theta_0 t/\sqrt\e},$$
for all $t$ in $[0,T]$, $m\ge 0$, and for some $\theta_0,c,C>0$. 

The proof is then by contradiction. That is, we assume that $\|T(s,t)\|_{\cL(H_w^m,L^2)} $ is bounded for all $0\le s\le t\le \e_0$, for some $\e_0>0$ and some $m\ge0$. We then introduce $u(t) \: := \:  T(0,t) u^\eps(0)$, and $\: v \: = \: u - u^\eps$, where $u^\eps$ is the growing solution defined above. The function $v$ then satisfies 
\begin{equation} \label{equationv}
 \pa_t   v +  \cL_s  v = - \e^Mr_\eps, \quad v\vert_{t=0} = 0,
 \end{equation} 
and thus obeys the standard Duhamel representation
$$ v(t) \: = \: - \e^M\int_0^t T(s,t) r^\eps(s) \, ds.$$ 
Thus, thanks to the bound on the $T(s,t)$ and the remainder $r^\eps$, we get that 
$$ \| v (t) \|_{L^2} \: \le \: C\e^M\, \int_0^t \| e^{y} r^\eps(s) \|_{H^m}(s) ds \: \le \:  \,  C \, \eps^{M+\frac{1}{2}-m} \,    
e^{\frac{\theta_0 t}{\sqrt{\eps}}}.  $$
Also, from the definition of $u(t)$, we have
$$ \| u(t)  \| _{L^2}  = \|T(0,t) u^\e(0)\|_{L^2}\: \le \: C \| e^{\alpha Y} \, u^\eps(0) \|_{H^m} \: \le \: C\, \eps^{-m}. $$
Combining these estimates together with the lower bound on $u^\e(t)$, we deduce
$$   C \, \eps^{-m} \: \ge \:   \| u(t)  \| _{L^2} \: \ge \:  \| u^\eps(t) \|_{L^2} - \| v (t) \|_{L^2}  \: \ge \: \left( c \, - \,  C \,
 \eps^{M+\frac{1}{2}-m} \right)  \, e^{\frac{\theta_0 t}{\sqrt{\eps}}}. $$
This then yields a contradiction for small enough $\eps$, $M$ large, and $t =K | \ln\eps | \sqrt{\eps}$ with a sufficiently large $K$. The theorem is therefore proved.     
\end{proof}

Next, we also recall the following uniqueness result for the linearized equation.
\begin{proposition}(\cite{GVN})\label{theo-uniquelin} 
Let $u_s=u_s(t,y)$ be a smooth shear flow satisfying 
$$
 \sup_{t\ge 0}\Big(\sup_{y\ge 0}|u_s|+\int_0^\infty y|\D_y u_s|^2dy\Big)< +\infty.
$$ Let $ u\in L^\infty(]0,T[; L^2(\T \times \RR_+))$ with $ \pa_y u\in  L^2(0,T \times \T \times  \RR_+)$ 
be a solution to the linearized equation of \eqref{prandtl-lin} around the shear flow, with $u\vert_{t=0} = 0$. Then, $u \equiv 0$. 
\end{proposition}
\begin{proof} For sake of completeness, we recall here the proof in \cite{GVN}. Let $ w\in L^\infty(]0,T[; L^2(\T \times \RR_+))$ and $ \pa_y w\in  L^2(0,T \times \T \times  \RR_+)$ be a solution to the linearized equation of \eqref{prandtl-lin} around the shear flow, with $w\vert_{t=0} = 0$. Let us define   $\hat w_k(t,y)$, $k\in \ZZ$, the Fourier transform of $w(t,x,y)$ in $x$ variable. We observe that for each $k$, $\hat w_k$ solves
\begin{equation}\label{eqs-wk}\left\{\begin{array}{rll}\D_t \hat w_k + ik u_s \hat w_k - ik \D_yu_s \int_0^y \hat w_k (y')dy' - \D_y^2 \hat w_k &=&0\\\hat w_k(t,0) &=& 0\\\hat w_k(0,y)&=&0.\end{array}\right.\end{equation}

Taking the standard inner product of the equation \eqref{eqs-wk} against the complex conjugate of $\hat w_k$ and using the standard Cauchy--Schwarz inequality to the term $\int_0^y\hat w_kdy'$, we obtain
$$\begin{aligned}\frac 12 \dt\|\hat w_k\|^2_{L^2(\RR_+)} + \|\D_y \hat w_k\|^2_{L^2(\RR_+)} &\le |k|\int_0^\infty|u_s||\hat w_k|^2 dy + |k|\int_0^\infty|\D_yu_s|y^{1/2}|\hat w_k| \|\hat w_k\|_{L^2(\RR_+)}dy\\&\le |k|\Big(\sup_{t,y}|u_s|+\int_0^\infty y|\D_yu_s|^2dy\Big)\|\hat w_k\|^2_{L^2(\RR_+)}.\end{aligned}$$
Applying the Gronwall lemma into the last inequality yields $$
\|\hat w_k(t)\|_{L^2(\RR_+)} \le C e^{C|k|t} \|\hat w_k(0)\|_{L^2(\RR_+)},$$ for some constant $C$.
Thus, $\hat w_k(t)\equiv 0$ for each $k\in \ZZ$ since $\hat w_k(0)\equiv 0$. That is, $w\equiv 0$, and the proposition is proved.
\end{proof}

\myskip

\section{No asymptotic boundary layer expansions}

In this section, we will disprove the nonlinear asymptotic boundary-layer expansion. Our proof is by contradiction and based on the linear ill-posedness result, Theorem \ref{theo-GD}. Indeed, for some $\gamma>0$, let us assume that expansion \eqref{exp-nu1} is valid with $\gamma>0$ in the Sobolev spaces in the sense of Definition \ref{def-valid-expansion} for any initial shear layer $U_s$ and initial data $\tu_0,\tv_0 \in e^{-\alpha Y}H^m(\TT\times \RR_+)$. That is, we can write 
\begin{equation}\label{exp-nu1}
\bmat u^\nu\\v^\nu\emat (t,x,y) =\bmat u^0 - u^0\vert_{y=0}\\ 0\emat(y)+\bmat u_{s}\\ 0\emat(t,\frac y{\sqrt\nu}) + (\sqrt\nu)^{\gamma} \bmat \tilde u^\nu\\ \sqrt \nu \tilde  v^\nu\emat(t,x,\frac y{\sqrt\nu}), \qquad \gamma>0,\end{equation}
and 
$$p^\nu (t,x,y) = (\sqrt\nu)^{\gamma}\tilde p^\nu(t,x,\frac y{\sqrt\nu}), \qquad \gamma>0,$$
where $\tu ^\nu,\tv^\nu \in L^\infty([0,T_\gamma]; H^1(\T \times \RR_+))$ and $\tilde p^\nu \in  L^\infty([0,T_\gamma]; L^2(\T \times \RR_+))$, for some $T_\gamma = T_{\gamma }(\|e^{\alpha Y}(U_s-U)\|_{H_{Y}^{m'}}, \|e^{\alpha Y}(
\tu_0,\tv_0)\|_{H_{x,Y}^{m}})$. Furthermore, for some $C_\gamma = C_{\gamma }(\|e^{\alpha Y}(U_s-U)\|_{H_{Y}^{m'}}, \|e^{\alpha Y}(
\tu_0,\tv_0)\|_{H_{x,Y}^{m}})$, 
there holds
\begin{eqnarray*}
\sup_{0\leq t\leq T_\gamma} \|(\tilde{u}^\nu(t),\tilde{v}^\nu
(t))\|_{H_{x,Y}^{1}} \leq C_{\gamma }(\|e^{\alpha Y}(U_s-U)\|_{H_{Y}^{m'}}, \|e^{\alpha Y}(
\tu_0,\tv_0)\|_{H_{x,Y}^{m}}).
\end{eqnarray*}
We then let $(\tu,\tv)$ and $\tilde p$ be the weak limits of $(\tu^\nu,\tv^\nu)$ and $\tilde p^\nu$ in $L^\infty([0,T_\gamma]; H^1(\T \times \RR_+))$ and in $ L^\infty([0,T_\gamma]; L^2(\T \times \RR_+))$, respectively, as $\nu \to 0$. Note that it is clear that $T_\gamma,C_\gamma$ are independent of the small parameter $\nu$.   

\myskip

Hence, plugging these expansions into \eqref{NSeqs}, we obtain 
$$\begin{aligned}
\D_t \tilde u^\nu + (u^0-u^0\vert_{y=0}+u_s) \D_x \tilde u^\nu &+ \tilde v^\nu (\sqrt \nu\dy u^0+\D_Y u_s)  + \D_x \tilde p^\nu - \D_Y^2 \tilde u^\nu \\
&= -(\sqrt\nu)^\gamma(\tilde u^\nu \D_x \tilde u^\nu + \tilde v^\nu \D_Y \tilde u^\nu) + \nu \D_x^2 \tilde u^\nu + \nu \D_y^2 u^0
\end{aligned}$$
and 
$$\begin{aligned}
\nu \D_t \tilde v^\nu + \nu (u^0-u^0\vert_{y=0}+u_s + (\sqrt\nu)^\gamma \tilde u^\nu)\D_x \tilde v^\nu + (\sqrt \nu)^{\gamma+2}\tilde v^\nu \D_Y \tilde v^\nu + \D_Y \tilde p^\nu = \nu^2 \D_x^2 \tilde v^\nu + \nu \D_Y^2 \tilde v^\nu 
\end{aligned}$$
We take $\nu \to 0$ in these expressions. Since $(\tu^\nu,\tv^\nu)(t)$ converges to $(\tu,\tv)$ weakly in $H^1$, the nonlinear terms 
$(\tilde u^\nu \D_x + \tilde v^\nu \D_Y )\tilde u^\nu$ and $(\tilde u^\nu \D_x + \tilde v^\nu \D_Y )\tilde v^\nu$ have their weak limits in $L^1$, and thus disappear in the limiting equations due to the factor of $(\sqrt{\nu})^\gamma$. Similar treatments hold for the linear terms. Note that $(u^0-u^0\vert_{y=0})(y) = y \dy u^0 = \sqrt \nu Y\dy u^0$ also vanishes in the limit.
We thus obtain the following equations for the limits in the sense of distribution: 
\begin{equation}\label{prandtl-lin1}
\begin{aligned}
\D_t \tu  +  u_s \D_x \tu +  \tv \D_Y  u_s - \D^2_Y  \tu  + \D_x \tilde p \;&=\; 0,
\\
\D_Y \tilde p \;&=\;0
\end{aligned}
\end{equation} 
and the divergence-free condition for $(\tu,\tv)$. From the second equation, $\tilde p= \tilde p(t,x)$. Setting $Y=+\infty$ in \eqref{prandtl-lin1} and noting that $(\tilde u,\tilde v)$ belong to the $H^1$ Sobolev space and $u_s$ has a finite limit as $Y\to+\infty$, we must get $\dx \tilde p \equiv 0$ in the distributional sense. That is, the next order in the asymptotic expansion solves the linearized Prandtl equation:\begin{equation}\label{prandtl-lin-1st}
\D_t \tu +\cL_s \tu\;=\;0, \qquad \tu\vert_{t=0} \;=\;u_0,
\end{equation} with zero boundary conditions at $Y=0$ and $Y=+\infty$, for arbitrary shear flow $u_s=u_s(t,Y)$. 

\myskip

Now, let $u_{s_0}$ be the shear flow in Theorem \ref{theo-GD} such that \eqref{GD-illposed} holds. Thus, we have that for a fixed $\e_0>0, m\ge0$, and any large $n$, there are $s_n,t_n$ with $0\le s_n\le t_n\le \e_0$ and a sequence of $u_0^n$ such that  
\begin{equation}\label{GD-illposed1}\|e^{\alpha Y}u^n_0\|_{H^{m+1}} = 1 \qquad \mbox{and}\qquad \|u_L^n(t_n)\|_{L^2} \ge 2n
 \end{equation} with $u_L^n(t)$ being the solution to the linearized equation \eqref{prandtl-lin} around $u_{s_0}$ with $u_L^n(s_n) = u_0^n$.

\myskip

For such a fixed shear flow $u_{s_0}$, and fixed $n$, $s_n$ as in \eqref{GD-illposed1}, we consider the expansion \eqref{exp-nu1} for $u_{s_n} = u_{s_0}(t+s_n)$ and initial data $(\tu^{\nu,n}_0,\tv^{\nu,n}_0)$ defined as 
\begin{equation}\label{initial-data-exp}
(\tu^{\nu,n}_0, \tv^{\nu,n}_0): = (u_0^n,v_0^n), \qquad \mbox{with}\quad v_0^n:= -\int_0^y \dx u_0^n dy'.\end{equation} 
We note that since $u_0^n$ is normalized, $(\tu^{\nu,n}_0, \tv^{\nu,n}_0)$ belongs to $e^{-\alpha Y}H^m$ with a finite norm of size independent of $n$. 
We let $T_\gamma,C_\gamma$ be the two continuous functions  and $(\tu ^{\nu,n},\tv^{\nu,n})$ be the corresponding solution in the expansion in $L^\infty([0,T_\gamma]; H^1(\T \times \RR_+))$ whose existence is guaranteed by the contradiction assumption, the Definition \ref{def-valid-expansion}. Furthermore, there holds 
\begin{equation}\label{uniformbound-n}
\sup_{0\leq t\leq T_\gamma} \|(\tilde{u}^{\nu,n}(t),\tilde{v}^{\nu,n}
(t))\|_{H_{x,Y}^{1}} \leq C_{\gamma }(\|e^{\alpha Y}(u_{s_0}(s_n)-U)\|_{H_{Y}^{m'}}, \|e^{\alpha Y}(u_0^n,v_0^n)\|_{H_{x,Y}^{m}}),
\end{equation}
in which $T_\gamma= T_{\gamma }(\|e^{\alpha Y}(u_{s_0}(s_n)-U)\|_{H_{Y}^{m'}}, \|e^{\alpha Y}(u_0^n,v_0^n)\|_{H_{x,Y}^{m}})$. We note that thanks to \eqref{GD-illposed1} and the fact that $u_{s_0}$ solves the heat equation \eqref{shear-profile}, the norms of the translated shear layer $u_{s_0}(s_n)-U$ and the initial data are independent of $n$ and $\nu$, and thus so are $T_\gamma$ and $C_\gamma$. In addition, since $\e_0$ was arbitrarily small so that \eqref{GD-illposed1} holds, we can assume that $\e_0 \le T_\gamma$.

Next, let $(\tu^n, \tv^n)$ be the limiting solutions of $(\tu ^{\nu,n},\tv^{\nu,n})$ when $\nu \to0$. As shown above, we then obtain the linearized Prandtl equation for $\tilde u^n$ with initial data $u_0^n$:
$$\D_t \tilde u^n +\cL_{s_n} \tilde u^n\;=\;0, \qquad \tilde u^n\vert_{t=0} \;=\;u^n_0.
$$Thus, if we define $u^n(t):=\tilde u^n(t-s_n)$, the above equation reads
$$\D_t u^n +\cL_{s_0}u^n\;=\;0, \qquad u^n\vert_{t=s_n} \;=\;u^n_0,
$$
which, by uniqueness of the linear flow, yields $u^n \equiv u_L^n$ on $[s_n,\e_0]$ and 
$$\| \tilde u^n(t_n-s_n)\|_{L^2} = \| u^n(t_n)\|_{L^2} \ge 2 n.$$
This implies that for small $\nu$, $\|\tu^{n,\nu}(t_n-s_n)\|_{L^2} \ge n$, which contradicts with the uniform bound \eqref{uniformbound-n} as $n$ is arbitrarily large and $C_\gamma$ is independent of $n$. The proof of Theorem \ref{theo-noexpansion} is complete.

\myskip

  \myskip
  
\section{Nonlinear ill-posedness 
} \label{sec-illposed}

Again, using the previous linear results, Theorem \ref{theo-GD}, we can prove Theorem \ref{theo-illposed} for the nonlinear equation \eqref{prandtl}. We proceed by contradiction. That is, we assume that the Cauchy problem \eqref{prandtl} is $(H^{m}, H^1)$ locally 
well-posed for some $m \ge 0$ in the sense of Definition \ref{def-wellposed}. Let $u_{s_0}$ be the fixed shear flow in Theorem \ref{theo-GD} such that \eqref{GD-illposed} holds. By definition, \eqref{GD-illposed} yields that for fixed $\e_0>0$ and any large $n$, there are $s_n,t_n$ with $0\le s_n\le t_n\le \e_0$ and a sequence of $u_0^n$ such that  
\begin{equation}\label{GD-illposed2}\|e^{\alpha Y}u^n_0\|_{H^m} = 1 \qquad \mbox{and}\qquad \|u_L^n(t_n)\|_{L^2} \ge n
 \end{equation} with $u_L^n(t)$ being the solution to the linearized equation \eqref{prandtl-lin} around $u_{s_0}$ with $u_L^n(s_n) = u_0^n$. We now fix $n$ large.

 \myskip
 
Next, define $v_0^{\delta,n} \: := \:  u_{s_0}(s_n)+ \delta u^n_0$, with $\delta$ a small parameter less than $\delta_0$. Let $v^{\delta,n}$ be the solution to the nonlinear equation \eqref{prandtl} with $v^{\delta,n}\vert_{t=0} = v_0^{\delta,n}$.  By the 
well-posedness applied to two solutions $v^{\delta,n}$ and the shear flow $u_{s_n}(t) := u_{s_0}(t+s_n)$, there are continuous functions $C(\cdot,\cdot),T(\cdot,\cdot)$ given in the definition such that
$$ \mbox{ess}\sup_{t \in [0,T]}\| v^{\delta,n}(t) -  u_{s_0}(t+s_n)\|_{H^1 } \: \le \: C \, \delta\|e^{\alpha Y}u^n_0\|_{H^m}  = C\,\delta,$$
in which $T = T(\|e^{\alpha Y}(v_0^{\delta,n} -U)\|_{H^m},\|e^{\alpha Y}(u_{s_0}(s_n) -U)\|_{H^m})$ and $C = C(\|e^{\alpha Y}(v_0^{\delta,n} -U)\|_{H^m},\|e^{\alpha Y}(u_{s_0}(s_n) -U)\|_{H^m})$. Thanks to \eqref{GD-illposed2} and the fact that $u_{s_0}$ solves the heat equation \eqref{shear-profile}, $v_0^{\delta,n} -U$ and $u_{{s_0}}(s_n)-U$ have their norms in $e^{-\alpha Y}H^m$ bounded uniformly in $n$ and $\delta$. Thus, the functions $T(\cdot,\cdot)$ and $C(\cdot,\cdot)$ can be taken independently of $n$ and $\delta$.  In what follows, we use $T,C$ for  $T(\cdot,\cdot), C(\cdot,\cdot)$.

In other words, the sequence $u^{\delta,n}:=\frac 1\delta(v^{\delta,n} -  u_{s_n})$ is  bounded in $L^\infty(0,T; H^1(\TT\times \RR_+))$  uniformly with respect to $\delta$, and moreover it solves 
\begin{equation}\label{eqs-w}
\D_t u^{\delta,n} + \cL_{s_n} u^{\delta,n} = \delta N(u^{\delta,n}), \qquad u^{\delta,n}(0) = u^n_0,
\end{equation} noting that $\cL_{s_n}$ is the operator linearized around the shear profile $u_{s_n}$ and $N$ is the nonlinear term: $N(u^{\delta,n}):=-u^{\delta,n}\dx u^{\delta,n} - v^{\delta,n}\dY u^{\delta,n}$. From the uniform bound on $u^{\delta,n}$, we deduce that, up to a subsequence, 
$$ u^{\delta,n} \rightarrow u^n \quad  L^\infty(0,T; H^1(\TT\times \RR_+)) \mbox{ weak$^*$  as } \delta \rightarrow 0. $$ 
We shall show that $u^n$ solves the linearized equation \eqref{prandtl-lin} in the sense of distribution. To see this, we only need to check with the nonlinear term. First, on any compact set $K$ of $\RR^+$, we obtain by applying the standard Cauchy inequality and using the divergence-free condition:
$$|v^{\delta,n}| \le \int_0^Y |\dx u^{\delta,n} |dY' \le C_0Y^{1/2}\Big(\int_{\RR^+} |\dx 
u^{\delta,n}|^2dY\Big)^{1/2},$$
and 
$$\begin{aligned} \int_{\TT\times K} |u^{\delta,n} v^{\delta,n}| dYdx&\le C_K\int_{\TT} \int_{K}|u^{\delta,n}|\Big(\int_{\RR^+} |\dx 
u^{\delta,n}|^2dY\Big)^{1/2} dY dx\\
&\le C_K\Big(\int_{\TT} \int_{K}|u^{\delta,n}|^2 dYdx\Big)^{1/2} \Big(\int_{\TT} \int_{\RR^+} |\dx 
u^{\delta,n}|^2dYdx\Big)^{1/2}\\
&\le C_K\|u^{\delta,n}\|_{H^1}^2
,\end{aligned}$$
for some constant $C_K$ depending on $K$. Now, from the divergence-free condition, we can rewrite $N(u^{\delta,n})$ as
$$N(u^{\delta,n}) = -\dx  (u^{\delta,n})^2 -\dY (u^{\delta,n} v^{\delta,n})$$
we have, for any smooth function $\phi$ that is compactly supported in $K$,
$$\begin{aligned} \delta \Big|\int_{\TT\times \RR_+} N(u^{\delta,n})  \phi dxdy \Big|&\le C_{K,\phi}\delta \int_{\TT\times K} \Big(|u^{\delta,n}|^2+ |u^{\delta,n} v^{\delta,n}|\Big)dxdY\\&\le C_{K,\phi}\delta \|u^{\delta,n}\|_{H^1}^2\longrightarrow 0,\end{aligned}$$
as $\delta \to 0$, thanks to the uniform bound on $u^{\delta,n}$ in $H^1$. Here, $C_{K,\phi}$ is some constant that depends on $K$ and $W^{1,\infty}$ norm of $\phi$. Thus, the nonlinearity $\delta N(u^{\delta,n})$ converges to zero in the above sense of distribution. This shows that by taking the limits of equation \eqref{eqs-w}, $u^n$ solves  
$$ \pa_t u^n + \cL_{s_n} u^n \: = \:  0, \quad u^n\vert_{t=0} = u^n_0.  $$

\myskip
By shifting the time $t$ to $t-s_n$, re-labeling $\tilde u^n(t) := u^n(t-s_n)$, and noting that by definition $\cL_{s_n}(t) = \cL_{s_0}(t+s_n)$, one has 
$$ \pa_t \tilde u^n + \cL_{s_0} \tilde u^n \: = \:  0, \quad \tilde u^n\vert_{t=s_n} = u^n_0,$$
that is, $\tilde u^n$ solves the linearized equation \eqref{prandtl-lin} around the shear flow $u_{s_0}$. By uniqueness of the linear flow (recalled in Proposition \ref{theo-uniquelin}), $\tilde u^n \equiv u^n_L$ on $[s_n,T]$. This therefore leads to a contradiction due to \eqref{GD-illposed2} and the fact that the bound for $u^{\delta,n}$ yields a uniform bound for $u^n$ and thus for $\tilde u^n$: 
$$ n\le  \|u_L^n(t_n)\|_{L^2}  =  \|\tilde u^n(t_n)\|_{L^2 } \le \sup_{t \in [s_n,T]}\|\tilde u^n(t)\|_{H^1} \: \le \: C,$$
for arbitrarily large $n$, upon recalling that $C$ is independent of $n$. This completes the proof of Theorem \ref{theo-illposed}. 


\myskip

\section{Well-posedness of the Oleinik's solutions}\label{sec-Ole}
In this section, we check that the Oleinik solutions to the Prandtl equation \eqref{prandtl} are well-posed in the sense of Definition \ref{def-wellposed}. Here, since now we only deal with the Prandtl equation, we shall write $(x,y)$ to refer $(x,Y)$ in \eqref{prandtl}, and use both $\D$ and subscripts whenever it is convenient to denote corresponding derivatives. To fit into the monotonic framework studied by Oleinik, we make the following assumption on the initial data and outer Euler flow: 

\myskip
(O) Assume that $U(t,x)$ is a smooth positive function and $\dx U,\dt U/U$ are bounded; the initial data $u_0(x,y)$ is an increasing function in $y$ with $u_0(x,0) =0$ and $u_0(x,y)\to U(0,x)$ as $y\to \infty$, and furthermore, for some positive constants $\theta_0,C_0$,  
\begin{equation}\label{decay-assump}\theta_0\;\le\; \frac{\dy u_0(x,y)}{U(0,x) - u_0(x,y)} \;\le\; C_0.\end{equation}
We also assume that all functions $\dy u_0, \dx u_0, \dx \dy u_0$ are bounded, and so are the ratios $\dy^2u_0/ \dy u_0 $ and $\dy^3u_0 \dy u_0 /  \dy^2 u_0$.  


\myskip

We now apply the Crocco change of variables: 
$$(t,x,y) \mapsto (t,x,\eta), \qquad \mbox{with}\quad \eta : = \frac{u(t,x,y)}{U(t,x)},$$
and the Crocco unknown function:
$$ w(t,x,\eta) := \frac{\dy u(t,x,y)}{U(t,x)}.$$
The Prandtl equation \eqref{prandtl} then yields
\begin{equation}\label{prandtl-Crocco}
\left\{
\begin{array}{rlll} 
\dt w + \eta U\dx w - A\D_\eta w - Bw
&=& w^2 \D^2_\eta w,  \quad &0<\eta<1,x\in \TT\\
(w \D_\eta w + \dx U + \dt U/U)\vert_{\eta=0}&=& 0, \\
w \vert_{\eta=1}&=& 0, 
\end{array}
\right.
\end{equation}
with initial conditions: $w\vert_{t=0} = w_0 = \D_y u_0/U$. Here, 
$$A:=(\eta^2-1)\dx U + (\eta -1)\frac{\dt U}{U}, \qquad B:= -\eta \dx U - \frac{\dt U}{U}.$$
To see how the boundary conditions are imposed, one notes that $\eta =0$ and $\eta =1$ correspond to the values at $y=0$ and $y=+\infty$, respectively. At $y=+\infty$, it is clear that $w = \D_y u =0$ since $u$ approaches to $U(t,x)$ as $y\to +\infty$, while by using the imposed conditions on $u$ and $v$ at $y=0$, we obtain from the equation \eqref{prandtl} that $$0=\D_y^2 u- \dx P=\D_y w + \dx U + \dt U/U= w\D_\eta w+ \dx U + \dt U/U.$$ 

\myskip

\begin{theorem}(\cite{Ole})\label{theo-Oleinik} Assume {\em (O)}. Then there exists a $T>0$ which depends continuously on the initial data such that the problems \eqref{prandtl-Crocco} and \eqref{prandtl} have a unique solution $w$ and $u$ on their respective domains, and there hold
\begin{equation}\label{bound-w} \theta_1 (1-\eta) \;\le\; w (t,x,\eta)\;\le\; \theta_2(1-\eta), \qquad |\dx w(t,x,\eta)| \;,\; |\dt w(t,x,\eta)|\;\le\; \theta_2(1-\eta)\end{equation}
for all $(t,x,\eta) \in [0,T]\times \TT\times (0,1)$, and 
\begin{equation}\label{bound-u} \theta_1 \;\le\; \frac {\dy u(t,x,y)}{U(t,x)-u(t,x,y)} \;\le \; \theta_2,\qquad e^{-\theta_2 y} \;\le\; 1-\frac{u(t,x,y)}{U(t,x)} \;\le\; e^{-\theta_1 y},\end{equation}
for all $(t,x,y) \in [0,T]\times \TT\times \RR_+$, for some positive constants $\theta_1,\theta_2$. In addition, weak derivatives $\dt u,\dx u,\dy\dx u,\dy^2 u,\dy^3 u$ are bounded functions in $[0,T]\times \TT\times \RR_+$. 
\end{theorem}

\begin{proof} In fact, the authors in \cite[Section 4.1, Chapter 4]{Ole} established the theorem in the case $x\in [0,X]$ with zero boundary conditions at $x=0$. Their analysis is based on the line method to discretize the $t$ and $x$ variables and to solve a set of second order differential equations in variable $\eta$. It is straightforward to check that these lines of analysis work as well in the periodic case $x \in \TT$ with minor changes in the choice of boundary conditions.  We thus omit to repeat the proof here. 
\end{proof}

\myskip
Using the estimates in Theorem \ref{theo-Oleinik}, we are able to prove that
\begin{theorem}\label{theo-wellposed}
The Cauchy problem \eqref{prandtl} under the assumption {\em (O)} is well-posed in the sense of Definition \ref{def-wellposed}, with some constant $\alpha$ and some continuous functions $T(\cdot,\cdot)$ and $C(\cdot,\cdot)$ appeared in the 
stability estimate \eqref{stab-ineq} that depend on $\theta_0,C_0$ in our assumption {\em (O)}. 
\end{theorem}

\myskip
In the proof, we need the following lemma.
\begin{lemma}\label{lem-L2stab-w}
Under the same assumptions as in Theorem \ref{theo-Oleinik}, we obtain 
\begin{equation}\label{I-t-est}
I(t) \le C I(0), \qquad 0\le t\le T,
\end{equation}
with $$I(t):= \int_{\TT\times [0,1]}\Big[ \frac{|w_{1x} - w_{2x}|^2}{(1-\eta)^\beta} + \frac{|w_{1} - w_{2}|^2}{(1-\eta)^\beta}\Big](t,x,\eta) dxd\eta  \Big], \qquad \forall 0\le \beta<3,
$$
for arbitrary two solutions $w_1,w_2$ to \eqref{prandtl-Crocco}. 
\end{lemma}
\begin{proof}[Proof of Lemma \ref{lem-L2stab-w}] We consider $w_1,w_2$ being solutions to \eqref{prandtl-Crocco}. We first note that $I(t)$ is well-defined for $\beta<3$ by the bounds in Theorem \ref{theo-Oleinik} that $|w_j|\le C(1-\eta)$ and $|w_{jx}|\le C(1-\eta)$. Let us introduce $\phi = w_1-w_2$. Then, $\phi$ solves
$$\left\{
\begin{array}{rlll} 
\phi_t + \eta U \phi_x - A\phi_\eta - B\phi - (w_1+w_2)\D_\eta^2w_2 \phi &=& w_1^2 \D^2_\eta \phi, \quad &0<\eta<1,x\in \TT\\
(w_1 \phi_\eta +w_{2\eta}\phi)\vert_{\eta=0}&=& 0, \\
\phi \vert_{\eta=1}&=& 0, 
\end{array}
\right.
$$
for $A,B$ being defined as in \eqref{prandtl-Crocco}. In particular, we have $|A|\le C(1-\eta)$ and $|B|\le C$. Multiplying the equation by $e^{-k\eta}\phi/(1-\eta)^\beta$ and integrating it over $\TT\times (0,1)$, we easily obtain 
$$\frac{1}{2}\frac{d}{dt} \int_{\TT\times (0,1)} \frac{e^{-k\eta}|\phi|^2}{(1-\eta)^\beta} dxd\eta \;=\;- \int_{\TT\times (0,1)} \Big[ \eta U \phi_x - A\phi_\eta - B\phi - (w_1+w_2)\D_\eta^2w_2 \phi - w_1^2\phi_{\eta\eta}\Big] \frac{e^{-k\eta}\phi}{(1-\eta)^\beta} dxd\eta.$$
We treat each term on the right-hand side. Using the bounds on $A,B$ and on $w_j,\D_\eta^2w_j$, it is easy to see that 
$$\begin{aligned}\Big|\int_{\TT\times (0,1)} \Big [  \eta U \phi_x -A\phi_\eta &- B\phi + (w_1+w_2)\D_\eta^2w_2 \phi \Big] \frac{e^{-k\eta}\phi}{(1-\eta)^\beta} dxd\eta \Big | \\\;&\le\; \epsilon \int_{\TT\times (0,1)}  \frac{e^{-k\eta}A^2 }{(1-\eta)^\beta} |\phi_\eta|^2 dxd\eta +C_\e\int_{\TT\times (0,1)} \frac{e^{-k\eta}|\phi|^2}{(1-\eta)^\beta} dxd\eta, \end{aligned}$$ for arbitrary small $\e$. 
For the last term, integration by parts yields
$$\begin{aligned} \int_{\TT\times (0,1)}  \frac{e^{-k\eta}w_1^2 }{(1-\eta)^\beta} \phi_{\eta\eta} \phi dxd\eta = &-
  \int_{\TT\times (0,1)}  \frac{e^{-k\eta}w_1^2 }{(1-\eta)^\beta} |\phi_\eta|^2 dxd\eta \\&-  \int_{\TT\times (0,1)} \D_\eta \Big(\frac{e^{-k\eta}w_1^2 }{(1-\eta)^\beta} \Big) \phi_\eta \phi dxd\eta - \int_{\TT\times \{\eta =0\}} \frac{e^{-k\eta}w_1^2 }{(1-\eta)^\beta} \phi_{\eta} \phi dx.
 \end{aligned}$$
Again, by integration by parts, we have
$$\begin{aligned} - \int_{\TT\times (0,1)} \D_\eta &\Big(\frac{e^{-k\eta}w_1^2 }{(1-\eta)^\beta} \Big) \phi_\eta \phi dxd\eta
= \frac 12\int_{\TT\times (0,1)} \D_\eta^2 \Big(\frac{e^{-k\eta}w_1^2 }{(1-\eta)^\beta} \Big) |\phi|^2 dxd\eta+ \frac12\int_{\TT\times \{\eta =0\}} \D_\eta\Big(\frac{e^{-k\eta}w_1^2 }{(1-\eta)^\beta}\Big)|\phi|^2dx.
\end{aligned}
$$
Thanks to the bounds $|w_j|\le C(1-\eta)$, we have 
$$\D_\eta^2 \Big(\frac{e^{-k\eta}w_1^2 }{(1-\eta)^\beta} \Big)\le C \frac{e^{-k\eta} }{(1-\eta)^\beta}.$$ 
Collecting all boundary terms, we need to estimate
$$\frac12\int_{\TT\times \{\eta =0\}} \Big[ - k\frac{e^{-k\eta}w_1^2 }{(1-\eta)^\beta}|\phi|^2 + \D_\eta\Big(\frac{w_1^2 }{(1-\eta)^\beta}\Big)e^{-k\eta}|\phi|^2-\frac{e^{-k\eta}w_1^2 }{(1-\eta)^\beta} \phi_{\eta} \phi \Big] dx.$$
Note that at $\eta=0$, $w_1\not=0$ and $w_1\phi_\eta = -w_{2\eta}\phi$. Thus, by taking $k$ sufficiently large in the above expression, we can bound it by 
$$- \frac k4\int_{\TT\times \{\eta =0\}} w_1^2|\phi|^2 dx.$$

\myskip
Combining the above estimates and choosing $\e$ sufficiently small, with noting that $|A|\le C(1-\eta)\le Cw_1$, we thus obtain 
\begin{equation}\label{L2-w-est}\begin{aligned} \frac{d}{dt} \int_{\TT\times (0,1)} \frac{e^{-k\eta}|\phi|^2}{(1-\eta)^\beta} dxd\eta &+  \int_{\TT\times (0,1)}  \frac{e^{-k\eta}w_1^2 }{(1-\eta)^\beta} | \phi_\eta|^2 dxd\eta \\&+\int_{\TT\times \{\eta =0\}} w_1^2|\phi|^2 dx\;\le\; C\int_{\TT\times (0,1)}  \frac{e^{-k\eta}|\phi|^2}{(1-\eta)^\beta}  dxd\eta.\end{aligned}\end{equation}

\myskip
To obtain estimates for $\phi_x$, we take $x$-derivative of the equation for $\phi$ and integrate the resulting equation over $\TT\times (0,1)$ against $e^{-k\eta}\phi_x/(1-\eta)^{\beta}$. We arrive at 
\begin{equation}\label{est-phix}\begin{aligned}\frac{1}{2}\frac{d}{dt} \int_{\TT\times (0,1)} &\frac{e^{-k\eta}|\phi_x|^2}{(1-\eta)^\beta} dxd\eta 
\\\;=\;&- \int_{\TT\times (0,1)} \Big[ \eta U \phi_{xx} +\eta U_x\phi_x- A_x\phi_\eta - A\phi_{x\eta} - B_x\phi-B\phi_x \\&- ((w_1+w_2)\D_\eta^2w_2)_x \phi  - 
(w_1+w_2)\D_\eta^2w_2\phi_x - w_1^2 \phi_{x\eta\eta} - 2w_1w_{1x}\phi_{\eta\eta}\Big] \frac{e^{-k\eta}\phi_x}{(1-\eta)^\beta} dxd\eta.\end{aligned}\end{equation}
 Similarly as in deriving the estimate \eqref{L2-w-est}, integration by parts and the bounds on $A,B, w_j$ easily yields
 \begin{equation}\label{est-1-phix}\begin{aligned} \int_{\TT\times (0,1)}  \Big[ w_1^2 \phi_{x\eta\eta} &+ 2w_1w_{1x}\phi_{\eta\eta} \Big] \frac{e^{-k\eta}\phi_x}{(1-\eta)^\beta} dxd\eta \\\le& -
 \frac12 \int_{\TT\times (0,1)}  \frac{e^{-k\eta}w_1^2 }{(1-\eta)^\beta} |\phi_{x\eta}|^2 dxd\eta + C \int_{\TT\times (0,1)} \frac{|\phi_x|^2+w_1^2|\phi_\eta|^2}{(1-\eta)^\beta}e^{-k\eta} dxd\eta \\&- \int_{\TT\times (0,1)} \D_\eta \Big[ \frac{e^{-k\eta}\phi_x}{(1-\eta)^\beta} \Big] \phi_{x\eta}\phi_xdxd\eta+ \int_{\TT\times \{\eta=0\}} \Big[ w_1^2 \phi_{x\eta} + 2w_1w_{1x}\phi_{\eta} \Big] \phi_x dx.\end{aligned}\end{equation}
 Here, we note that there is a crucial factor of $w_1^2$ in front of the term $|\phi_\eta|^2$ thanks to the bounds: $w_j \sim (1-\eta)$ and $|w_{jx}|\le C(1-\eta)$. Again, applying integration by parts to the third term on the right-hand side yields
$$\begin{aligned}- \int_{\TT\times (0,1)} \D_\eta &\Big[ \frac{e^{-k\eta}w_1^2}{(1-\eta)^\beta} \Big] \phi_{x\eta}\phi_xdxd\eta 
\\&= \frac12 \int_{\TT\times (0,1)} \D^2_\eta \Big[ \frac{e^{-k\eta}w_1^2}{(1-\eta)^\beta} \Big] |\phi_{x}|^2dxd\eta + \int_{\TT\times\{\eta=0\}} \D_\eta \Big[ \frac{e^{-k\eta}w_1^2}{(1-\eta)^\beta} \Big] |\phi_x|^2dx,\end{aligned}$$ 
 where the last boundary term is clearly bounded by $$-\frac k2  \int_{\TT\times\{\eta=0\}} w_1^2|\phi_x|^2dx.$$ 
 We now estimate the boundary term in \eqref{est-1-phix}. We recall that at the boundary $\eta=0$, we have $w_1\phi_\eta = - w_{2\eta}\phi$. Thus, 
 $$w_1^2\phi_{x\eta} = w_1(-w_{2\eta}\phi_x - w_{2x\eta}\phi - w_{1x}\phi_\eta) =-w_1(w_{2\eta}\phi_x + w_{2x\eta}\phi ) + w_{1x}w_{2\eta}\phi. $$
 That is, the normal derivative $\phi_\eta$ on the boundary can always be eliminated to yield    
  $$\int_{\TT\times \{\eta=0\}} \Big[ w_1^2 \phi_{x\eta} + 2w_1w_{1x}\phi_{\eta} \Big] \phi_x dx \le C\int_{\TT\times\{\eta=0\}} (|\phi|^2+|\phi_x|^2)dx.$$ 

\myskip
The remaining terms on the right-hand side of \eqref{est-phix} are again easily bounded by 
$$C \int_{\TT\times (0,1)} e^{-k\eta} \frac{|\phi|^2+|\phi_x|^2}{(1-\eta)^\beta} dxd\eta. $$

\myskip
Putting these estimates into \eqref{est-phix}, we have obtained 
\begin{equation}\label{L2-wx-est}\begin{aligned}\frac{d}{dt} \int_{\TT\times (0,1)} &\frac{e^{-k\eta}|\phi_x|^2}{(1-\eta)^\beta} dxd\eta 
\;\le\;
C \int_{\TT\times (0,1)} \frac{|\phi|^2+|\phi_x|^2 + |w_1|^2|\phi_\eta|^2}{(1-\eta)^\beta} dxd\eta 
+ C\int_{\TT\times\{\eta=0\}}|\phi|^2dx
.\end{aligned}\end{equation}

\myskip
Adding together this inequality with a large constant $M$ times the inequality \eqref{L2-w-est}, we can get rid of the boundary term and the term involving $|\phi_\eta|^2$ on the right-hand side of \eqref{L2-wx-est} and thus obtain 
\begin{equation}\label{L2-wx-final}\begin{aligned}\frac{d}{dt} \int_{\TT\times (0,1)} &e^{-k\eta}\frac{M|\phi|^2+|\phi_x|^2}{(1-\eta)^\beta} dxd\eta 
\;\le\;
C(M) \int_{\TT\times (0,1)} e^{-k\eta}\frac{M|\phi|^2+|\phi_x|^2}{(1-\eta)^\beta} dxd\eta
.\end{aligned}\end{equation}

\myskip
The claimed estimate \eqref{I-t-est} thus immediately follows from \eqref{L2-wx-final} by the standard Gronwall inequality, and this completes the proof of Lemma \ref{lem-L2stab-w}. 
\end{proof}

\myskip
We are now ready to give
\begin{proof}[Proof of Theorem \ref{theo-wellposed}] We only need to check the 
stability estimate \eqref{stab-ineq}. Let $U(t,x)$ be a fixed Euler flow, and take $u_{01}(x,y)$ and $u_{02}(x,y)$ be arbitrary smooth functions satisfying the assumption (O). Let $u_1,u_2$ be solutions to \eqref{prandtl} and $w_1,w_2$  the corresponding solutions to \eqref{prandtl-Crocco} constructed by Theorem \ref{theo-Oleinik}. Set $z = u_1 - u_2$ and $h = v_1-v_2$ with $v_j$ being determined through the divergence-free condition with $u_j$. Then, $z$ and $h$ solve 
\begin{equation}\label{prandtl-zeqs}\dt z + u_1 \dx z + z \dx u_2 + v_1 \dy z + h \dy u_2 = \dy^2 z, \qquad h = -\int_0^y \dx z dy',\end{equation}
with $z\vert_{y=0} = z\vert_{y=+\infty}=0$. 

\myskip
Multiplying the equation for $z$ by $e^{-kt}z$ for some large $k$, taking integration over $\TT\times \RR_+$, and applying integration by parts, we obtain 
\begin{equation}\label{ineq1} \frac 12 \frac{d}{dt}\int_{\TT\times \RR_+} |z|^2 dxdy + \int_{\TT\times \RR_+} \Big[(k+ \dx u_2) |z|^2 + \dy u_2 hz + |\dy z|^2\Big]dxdy \;=\; 0.\end{equation}
By the definition of $h$, we can estimate
$$ \begin{aligned}\Big|\int_{\TT\times \RR_+} \dy u_2 hz dxdy\Big| 
&= \Big| \int_{\TT\times \RR_+} \dy (u_2 - U) z \Big(\int_0^y \dx z dy' \Big)dxdy\Big| 
\\
&\le \sup_{t,x}  \Big(\int_{\RR_+}y^{1/2}\dy(u_2 - U) dy\Big) \| z\| \|\dx z\|
\end{aligned}$$
for some $\alpha<1/2$, where $\|\cdot\|$ denotes the standard $L^2$ norm on $\TT\times \RR_+$. Thanks to bounds \eqref{bound-u}, $u_2$ converges exponentially to $U$ as $y\to \infty$ and thus the integral $\int_{\RR_+}y^{1/2}\dy(u_2 - U) dy$ is finite. In addition, since the derivatives $\dx u_j, \dy u_j$ are bounded, by taking $k$ sufficiently large, the identity \eqref{ineq1} yields
\begin{equation}\label{L2-est-z} \frac{d}{dt}\int_{\TT\times \RR_+} |z|^2 dxdy + \int_{\TT\times \RR_+} \Big[ |z|^2+ |z_y|^2\Big] dxdy \;\le\;  
C\|z_x\|^2.\end{equation} 

\myskip
We will next derive estimates for $z_y$. For this, we take derivative with respect to $y$ to the equation for $z$ and multiply the resulting equation by $\dy z$. With noting that $z\vert_{y=0}=0$ and $z_{yy}\vert_{y=0}=0$ (obtained by setting $y=0$ in \eqref{prandtl-zeqs}), easy computations yield 
$$\begin{aligned}\frac12\frac{d}{dt}& \int_{\TT\times \RR_+} |z_y|^2dxdy
+  \int_{\TT\times \RR_+} |z_{yy}|^2 dxdy+ \int_{\TT\times \RR_+} \Big[\frac 12 u_1 \dx |z_y|^2 
\\&+ (v_{1y}+u_{2x}) |z_y|^2 +u_{1y}z_xz_y+ u_{2y}h_yz_y+ u_{2xy}zz_y + v_1\frac12\dy |z_y|^2 +u_{2yy}hz_y\Big]dxdy \;=\; 0.\end{aligned}$$ 
Again, by using the boundedness of $u_{jx},u_{jxy},u_{jyy}$, the divergence-free condition $h_y=-z_x$, and similar estimates on the term involving $h$ as above, we easily get 
\begin{equation}\label{L2-est-zy} \frac{d}{dt}\int_{\TT\times \RR_+} |z_y|^2 dxdy  \;\le\;   
C \Big( \|z\|^2 + \|z_y\|^2 + \| z_x\|^2\Big).\end{equation} 
We note that by using the fact that the derivatives $u_{jx},u_{jxy},u_{jyy}$ are not only bounded, but also decay exponentially in $y$, similar estimates as done above also yield
\begin{equation}\label{L2-est-yzy} \frac{d}{dt}\int_{\TT\times \RR_+} y^n|z_y|^2 dxdy  \;\le\;   
C \Big( \|z\|^2 + \|z_y\|^2 + \| z_x\|^2\Big), \qquad \forall n\ge 0.\end{equation} 

\myskip
Finally, we may wish to give similar estimates for $z_x$. That is, taking $x$-derivative to the equation for $z$, testing the resulting equation by $z_x$, and using the boundary condition ${z_x}\vert_{y=0}=0$, one may get
\begin{equation}\label{ineq3} \begin{aligned}\frac{d}{dt}& \int_{\TT\times \RR_+} |z_x|^2dxdy+  \int_{\TT\times \RR_+} |z_{xy}|^2 dxdy\\&+ \int_{\TT\times \RR_+} \Big[(u_{1x}+u_{2x}) |z_x|^2 
+ u_{2xx}zz_x + v_{1x}z_xz_y + u_{2y}h_xz_x+u_{2xy}hz_x\Big]dxdy \;=\; 0.\end{aligned}\end{equation} 
However, it is not at all immediate to estimate the term $u_{2y}h_xz_x$ in the above identity to yield a similar bound as in \eqref{L2-est-zy} since $h$ has the same order as $z_x$ by its definition (see \eqref{prandtl-zeqs}). 

\myskip
Therefore, we shall derive estimates for $z_x$ through the equation \eqref{prandtl-Crocco} and the estimates on $w$ obtained in Lemma \ref{lem-L2stab-w}. First, we recall that $u$ is defined through $w$ by the relation (see, for example, \cite[Eq. (4.1.52)]{Ole}):
$$ y = \int_0^{u(t,x,y)/U(t,x)} \frac{1}{w(t,x,\eta')} d\eta'.$$
Differentiating this identity with respect to $x$, we immediately obtain{\footnote{There is an unfortunate typo in \cite[Eq. (4.1.53)]{Ole} where the integral in \eqref{ux-expression} was $\int_0^{u/U}\frac{w_x}{w}(t,x,\eta')d\eta'$.}} 
\begin{equation}\label{ux-expression} u_x = u\frac{U_x}{U} + wU \int_0^{u/U}\frac{w_x}{w^2}(t,x,\eta')d\eta',\end{equation}
for $u$ and $w$ being solutions to \eqref{prandtl} and \eqref{prandtl-Crocco}. We apply this expression to $u_1,w_1$ and $u_2,w_2$, respectively and derive an estimate for $z_x = u_{1x}-u_{2x}$.  
In regions where $u_1\ge u_2$, it will appear to be convenient to estimate $z_x$ as follows:
\begin{equation}\label{zx-estimate1}\begin{aligned}|z_x| 
\le &C\Big[|z| + |w_1(t,x,u_1/U)-w_2(t,x,u_2/U)|  \int_0^{u_2/U}\Big|\frac{w_{2x}}{w^2_2}\Big|(t,x,\eta')d\eta' 
\\
&+ |w_1|\Big|\int_{u_1/U}^{u_2/U}\frac{w_{1x}}{w_1^2}(t,x,\eta')d\eta'\Big| + |w_1|\int_0^{u_2/U}\Big|\frac{w_{1x}}{w_1^2} - \frac{w_{2x}}{w_2^2}\Big|(t,x,\eta')d\eta'  \Big].
\end{aligned}
\end{equation}
Whereas in regions where $u_1\le u_2$ we estimate 
\begin{equation}\label{zx-estimate2}\begin{aligned}|z_x| 
\le &C\Big[|z| + |w_1(t,x,u_1/U)-w_2(t,x,u_2/U)|  \int_0^{u_1/U}\Big|\frac{w_{1x}}{w^2_1}\Big|(t,x,\eta')d\eta' 
\\
&+ |w_2|\Big|\int_{u_1/U}^{u_2/U}\frac{w_{2x}}{w_2^2}(t,x,\eta')d\eta'\Big| + |w_2|\int_0^{u_1/U}\Big|\frac{w_{1x}}{w_1^2} - \frac{w_{2x}}{w_2^2}\Big|(t,x,\eta')d\eta'  \Big].
\end{aligned}
\end{equation}
From the definition $w_j(t,x,u_j/U) = \dy u_j(t,x,y)$, we have $|w_1(t,x,u_1/U)-w_2(t,x,u_2/U)|= |z_y|$. Also, note that $|w_{jx}/w_j|$ is uniformly bounded. We have $$ \int_0^{u_j/U}\Big|\frac{w_{jx}}{w^2_j}\Big|(t,x,\eta')d\eta' \;\le\; C  \int_0^{u_j/U}\Big|\frac{1}{w_j}\Big|(t,x,\eta')d\eta' = Cy,$$
and 
$$|w_j|\Big|\int_{u_1/U}^{u_2/U}\frac{w_{jx}}{w_j^2}(t,x,\eta')d\eta'\Big| \le C|w_j|\Big|\int_{u_1/U}^{u_2/U}\frac{1}{1-\eta'}d\eta' \Big|\le C\Big(\frac{|w_j|}{1-u_1/U} + \frac{|w_j|}{1-u_2/U}\Big)|u_1-u_2|.$$
Now, if $u_1\ge u_2$, we use the estimate \eqref{zx-estimate1} and the fact that $|w_j|\le C(1-u_j/U)$. We thus obtain 
$$\frac{|w_1|}{1-u_1/U} + \frac{|w_1|}{1-u_2/U} \le 2 \frac{|w_1|}{1-u_1/U} \le C.$$
Similarly, if $u_1\le u_2$, we use \eqref{zx-estimate2} and replace $w_1$ by $w_2$ in the above inequality, leading to the similar uniform bound. This explains our choice of expressions in \eqref{zx-estimate1}-\eqref{zx-estimate2}. By combining these estimates, the second and third terms in \eqref{zx-estimate1} when $u_1\ge u_2$ and in \eqref{zx-estimate2} when $u_1\le u_2$ are bounded by 
$$C(|z| + y|z_y|).$$

\myskip
Finally,  we give estimates for the last term in inequalities \eqref{zx-estimate1} and \eqref{zx-estimate2}. Using the estimates on $w,w_x$, we have  
$$\Big|\frac{w_{1x}}{w_1^2} - \frac{w_{2x}}{w_2^2}\Big| \le C \frac{|w_{1x} - w_{2x}|}{(1-\eta')^2} + C\frac{|w_1-w_2|}{(1-\eta')^2}, \qquad \forall \eta'\in (0,1),$$
which together with the standard H\"older inequality implies that 
$$\begin{aligned}\int_{\TT\times \RR_+} |w_j|^2\Big|\int_0^{u_k/U}&\Big(\frac{w_{1x}}{w_1^2} - \frac{w_{2x}}{w_2^2}\Big)(t,x,\eta')d\eta'\Big|^2 dxdy 
\\&
\le C\sup_{t,x}\int_{\RR_+} |\dy u_j|^2 (1-\frac{u_k}{U})^{\beta-2} dy  \int_{\TT\times [0,1]}\Big[ \frac{|w_{1x} - w_{2x}|^2}{(1-\eta')^\beta} + \frac{|w_{1} - w_{2}|^2}{(1-\eta')^\beta}\Big]dxd\eta'
\\&
\le C\sup_{t,x}\int_{\RR_+} e^{-2\theta_1 y} e^{(3-\beta)\theta_2y} dy  \int_{\TT\times [0,1]}\Big[ \frac{|w_{1x} - w_{2x}|^2}{(1-\eta')^\beta} + \frac{|w_{1} - w_{2}|^2}{(1-\eta')^\beta}\Big]dxd\eta'
\\&
\le C\int_{\TT\times [0,1]}\Big[ \frac{|w_{1x} - w_{2x}|^2}{(1-\eta')^\beta} + \frac{|w_{1} - w_{2}|^2}{(1-\eta')^\beta}\Big]dxd\eta',
\end{aligned}$$
for some $\beta<3$ satisfying $(3-\beta)\theta_2 \le \theta_1$.  

\myskip
Thus, we have obtained 
\begin{equation}\label{L2-est-zx}\|z_x\|^2_{L^2} \le C\Big[\|z\|^2_{L^2} + \|yz_y\|^2_{L^2} +  \int_{\TT\times [0,1]}\Big[ \frac{|w_{1x} - w_{2x}|^2}{(1-\eta)^\beta} + \frac{|w_{1} - w_{2}|^2}{(1-\eta)^\beta}\Big]dxd\eta  \Big],\end{equation} for some $\beta<3$. Now, applying Lemma \ref{lem-L2stab-w} into \eqref{L2-est-zx}, we then have the following estimate: 
\begin{equation}\begin{aligned}
\|z_x\|^2_{L^2} 
&\le C\Big[\|z\|^2_{L^2} + \|yz_y\|^2_{L^2} +  \int_{\TT\times [0,1]}\Big[ \frac{|w_{1x} - w_{2x}|^2}{(1-\eta)^\beta} + \frac{|w_{1} - w_{2}|^2}{(1-\eta)^\beta}\Big](0,x,\eta)dxd\eta  \Big].
\end{aligned}\end{equation} 

\myskip
Combining this with estimates \eqref{L2-est-z}, \eqref{L2-est-zy}, and \eqref{L2-est-yzy} and applying the standard Gronwall's inequality, we easily obtain 
\begin{equation}\label{L2-final} \|z\|^2_{H^1}(t) \le  C(T) \Big[\|z_0\|^2_{H^1} +\|yz_{0y}\|^2_{L^2}+ \int_{\TT\times [0,1]}\Big[ \frac{|w_{1x} - w_{2x}|^2}{(1-\eta)^\beta} + \frac{|w_{1} - w_{2}|^2}{(1-\eta)^\beta}\Big](0,x,\eta)dxd\eta \Big], \end{equation}
where we have denoted $z_0 = u_{01}-u_{02}$. 

\myskip
Note that $\|yz_{0y}\|^2_{L^2} \le \|e^yz_{0y}\|^2_{L^2}$. It thus remains to express the last estimate in terms of initial data $u_{01}$ and $u_{02}$. We note that for $\eta = u_1(0,x,y)/U(t,x)$, $$\begin{aligned}|w_{1} - w_{2}|(0,x,\eta) &\le |w_1(0,x,u_1/U)-w_2 (0,x,u_2/U)| + |w_2(0,x,u_1/U) - w_2(0,x,u_2/U)|
\\&\le |\dy (u_1-u_2)(0,x,y)| + |\D_\eta w_2| |u_1 - u_2|(0,x,y)
.\end{aligned}$$
In addition, for $\eta = u_1(0,x,y)/U(t,x)$, assumptions on initial data (see (O)) gives $(1-\eta)^{-1}\le Ce^{\theta_2y}$ and $|\eta_y| = |\dy u_{01}/U|\le C(1-u_{01}/U)$. Thus, we can make change of variable $\eta$ back to $y$ and estimate 
$$\begin{aligned}
\int_{\TT\times [0,1]}\frac{|w_{1} - w_{2}|^2}{(1-\eta)^\beta}(0,x,\eta)dxd\eta 
&\;\le\; C \int_{\TT\times \RR_+} e^{(\beta-1)\theta_2y} (|\dy(u_{01}-u_{02})|^2+ |u_{01} - u_{02}|^2)(x,y)\; dxdy
\\&\;\le\; C \|e^{(\beta-1)\theta_2y/2}(u_{01}-u_{02})\|_{H^1}^2.\end{aligned}
 $$
 Similarly, we have
$$\begin{aligned}|w_{1x} - w_{2x}|(0,x,\eta) &\;\le\; |\dx\dy(u_{01}-u_{02})| (x,y)+ C|\dx(u_{01} - u_{02})|(x,y), 
\end{aligned}$$
and thus 
$$\begin{aligned}
\int_{\TT\times [0,1]}\frac{|w_{1x} - w_{2x}|^2}{(1-\eta)^\beta}(0,x,\eta)dxd\eta 
\;\le\; C \|e^{(\beta-1)\theta_2y/2}(u_{01}-u_{02})\|_{H^2}^2.\end{aligned}
 $$

\myskip
Putting these into \eqref{L2-final}, we have obtained 
\begin{equation}\begin{aligned}
\|(u_1-u_2)(t)\|^2_{H^1} 
&\;\le\; C\|e^{\alpha y} (u_{01} - u_{02}) \|_{H^2}^2,
\end{aligned}\end{equation} 
for $\alpha = (\beta-1)\theta_2/2$. Theorem \ref{theo-wellposed} thus follows. 
\end{proof}

\section{Acknowledgements} This work is grown out of the previous joint work with David G\'erard-Varet \cite{GVN}, and the second author greatly thanks him for many fruitful discussions. Y. Guo's research is supported in part by a NSF grant DMS-0905255 and a Chinese NSF grant 10828103.


\end{document}